\documentclass[conference]{IEEEtran}
\IEEEoverridecommandlockouts

\usepackage{amsmath,amssymb,amsthm,amsfonts}
\usepackage{optidef}         
\usepackage{mathtools}
\usepackage{algorithmic}
\usepackage{graphicx}
\usepackage{subcaption}
\usepackage{textcomp}
\usepackage{xcolor}
\usepackage{adjustbox}
\usepackage[noadjust]{cite}

\def\BibTeX{{\rm B\kern-.05em{\sc i\kern-.025em b}\kern-.08em
    T\kern-.1667em\lower.7ex\hbox{E}\kern-.125emX}}

\newtheorem{assumption}{Assumption}
\newtheorem{definition}{Definition}
\newtheorem{lemma}{Lemma}
\newtheorem{theorem}{Theorem}
\newtheorem{remark}{Remark}

\title{Networked Control and Mean Field Problems Under Diagonal Dominance: Decentralized and Social Optimality}
\author{Vivek Khatana$^{\dagger,\ddagger,\ast}$, Duo Wang$^{\dagger,\ddagger,\ast}$, Petros Voulgaris$^{\dagger}$, Nicola Elia$^{\star}$,
Naira Hovakimyan$^{\ddagger}$
\thanks{$^{\ast}$ Equal contribution}%
\thanks{$^{\dagger}$ The authors are with the Department of Mechanical Engineering, University of Nevada, Reno, NV-USA 
        {\tt\small \{vkhatana, duow,  pvoulgaris\}@unr.edu}}
\thanks{$^{\star}$ The author is with the Electrical and Computer Engineering Department,
University of Minnesota Twin Cities, MN, USA {\tt\small neli@umn.edu}}
\thanks{$^{\ddagger}$ The authors are with the Department of Mechanical Science and Engineering, Grainger College of Engineering, University of Illinois at Urbana-Champaign, Urbana, IL-USA {\tt\small \{vkhatana, duowang, nhovakim\}@illinois.edu}}
\thanks{This work was supported in part by NSF CMMI Award $2137764$, NASA
under the cooperative agreement $80$NSSC$20$M$0229$ and University Leadership Initiative Grant $80$NSSC$22$M$0070$ and NSF ECPN Award $2311007$.}
}

\begin{document}
\maketitle

\begin{abstract}
In this article, we employ an input-output approach to expand the study of cooperative multi-agent control and optimization problems characterized by mean-field interactions that admit decentralized and selfish solutions. The setting involves $n$ independent agents that interact solely through a shared cost function, which penalizes deviations of each agent from the group’s average collective behavior.
Building on our earlier results established for homogeneous agents, we extend the framework to nonidentical agents and show that, under a diagonal dominant interaction of the collective dynamics, with bounded local open-loop dynamics, the optimal controller for ${\cal H}_\infty$ and ${\cal H}_2$ norm minimization remain decentralized and selfish in the limit as the number of agents $n$ grows to infinity. 
\end{abstract}


\section{Introduction}
The study of networked control systems is an active area of research in recent years, with researchers addressing questions pertaining to various aspects such as structural properties like controllability/observability~\cite{liu2011controllability, rahmani2009controllability, pasqualetti2014controllability}, 
performance, or noise and uncertainty amplification~\cite{bamieh2012coherence, siami2013fundamental},
distributed controller design~\cite{bamieh2002distributed,bamieh2005convex, vamsi_elia_tac,lall_rotkowitz} among others. In this paper, we focus on the design of optimal controllers for a collection of $n$ decoupled multi-agent systems that are coupled through a shared social cost. This cost is defined as an input-output performance metric of the network, quantified by the deviation of each agent’s output from the collective average output. Additionally, an extra term accounts for individual control effort, which we incorporate as a constraint in part of the analysis. The overall formulation is closely related to frameworks considered in Mean Field (MF) games~\cite{caines_lqg0, huang_lqg, emilio, caines_lqg, Lions, nourian, basar, astolfi}.

Generally, in MF formulations (see, e.g., \cite{caines_lqg}), decentralized strategies are typically derived by substituting the true average measurements with a deterministic signal that characterizes the aggregate behavior. Under suitable assumptions, this approximation converges, asymptotically to the expected value of the average measurement signals and can be computed locally by each agent. Computation of the MF term generally requires solving a set of differential equations in advance, which may be challenging and computationally demanding~\cite{carmona}. Furthermore, most MF game methodologies adopt state-space, stochastic models, and many are formulated with finite-horizon cost criteria.

Although various input-output methodologies for distributed control have been proposed (e.g., \cite{voulgaris,qisalvoukha04,lall_rotkowitz,vamsi_elia_tac,cedric_andrea}), their scalability and practical realizability for a large number of agents are not straightforward in the problem settings considered in these works. In contrast, the MF formulation that we consider has a particular structure that enables explicit solutions that are both scalable and stably implementable.

In this article, we present a unified and systematic framework that both broadens and deepens our earlier contributions to networked control~\cite{voueli17journal, ACC18, IFAC17,tac21, voulgaris2025decentralized}. Our most recent work~\cite{voulgaris2025decentralized} advanced prior results on homogeneous agents~\cite{voueli17journal, ACC18, IFAC17,tac21} by establishing that, even in the heterogeneous setting, decentralized and selfish strategies achieve optimal $\mathcal{H}_\infty$ and $\mathcal{H}_2$ norm minimization asymptotically as $n \to \infty$, provided the collective input-output map remains uniformly bounded. By decentralized and selfish, we mean that each agent can disregard its deviation from the collective average—the social coupling—and instead optimize its own locally regulated objectives solely using local information. In other words, for infinite-horizon problems, the MF terms typically required in MF game formulations become unnecessary.

The present article significantly expands these results: we demonstrate that optimality of the decentralized and selfish solutions persists even under a substantially weaker condition—namely, \emph{when the collective input-output dynamics satisfy a diagonal dominance property with only the local input-output maps of agents being bounded}. In this relaxed setting, we prove that $\mathcal{H}_\infty/\mathcal{H}_2$ norm minimization continues to admit optimal decentralized controllers as the number of agents grows. Note that diagonal dominance appears in many large-scale systems. In power networks, local inverter dynamics dominate over weaker couplings through admittances~\cite{dorfler2013synchronization, simpson2013synchronization}, in traffic and platooning, vehicle actuation outweighs spacing and alignment effects~\cite{seiler2004disturbance, bamieh2012coherence}, in multi-robot swarms, stabilization is primarily local with only mild aggregate coupling~\cite{jadbabaie2003coordination,ADMM_tac}, and in social networks, agents are driven mainly by individual preferences with interactions mediated through aggregate quantities~\cite{basar1999dynamic, huang2006large}. These examples highlight that diagonal dominance captures a broad class of practical systems where decentralized strategies are desired.

\subsection{Notations And Definitions}
For a real sequence $M=\{M_k\}_{k=0}^\infty$ we use the $\ell_2$ or ${\cal H}_2$ norm $\left\|M\right\|_2:=[\sum_k M_k^2]^{1/2}$.  For a real sequence of matrices $M=[M_{ij}]=\{M_k\}_{k=0}^\infty$ we use the $\ell_2$ or ${\cal H}_2$ norm $\left\|M\right\|_2:=[\sum_{i,j} \left\|M_{ij}\right\|_{2}^{2}]^{1/2} $.  If $M$ is a transfer function $\left\|M\right\|:=\sup_{\omega}\sigma_{\rm max}[M(e^{\jmath\omega})]$, where $\sigma_{\rm max}$ stands for the maximum singular value and $M(\lambda)=\sum_{k= 0}^{\infty}M_k\lambda^k$ is the $\lambda$-transform; we note that this is the ${{\cal H}_\infty}$ norm of $M$ which is the $\ell_2$-induced norm of the map $M$. $\mathbf{1}$ and $I$ denote a $n$-dimensional vector with all entries equal to $1$ and the $n$-dimensional identity matrix respectively. Denote $\mathbf{T} := \left( I - \frac{1}{n}\mathbf{1}\mathbf{1}^\top \right)$ and $\text{diag}(.)$ is a matrix operator that creates a diagonal matrix created by putting the arguments of the input on the diagonal. 

\begin{definition}
    Let $f$ and $g$ be real-valued functions, both defined on some unbounded subset of the positive real numbers, we say $f$ is little-o of $g$, i.e. $f(x) = o(g(x))$ if for every positive constant $\eta$ there exists a constant $x_0$ such that $|f(x)| \leq \eta g(x)$ for all $x \geq x_0$. Further, we say $f$ is big-O of $g$, i.e., $f = O(g(x))$ if there exists a positive real number $B$ and a real number $x_0$ such that $|f(x)| \leq B |g(x)|$ for all $x \geq x_0$.
\end{definition}

\begin{definition}
    A $n \times n$ matrix $Q$ is called $\alpha$-column diagonal dominant with $\alpha > 0$, if 
    \begin{align}
        \alpha \|Q_{jj}\| \geq  \sum_{i=1, i \neq j}^n \|Q_{ij}\|,
    \end{align}
    for all $j = 1,\dots,n$. 
\end{definition}

\section{Problem Setup}
We consider $n$ dynamically decoupled systems $\{G_i\}_{i=1}^n$ where each $G_i$ has a control input $u_i$, a measurement output $y_i$, a disturbance $w_i$ and a regulated variable $z_i$. Throughout this article, we deal with $n$ agents unless otherwise specified. Let $z = \Phi w$ with the vectors of regulated and disturbance signals $z = [z_i]_{1 \leq i \leq n}, w = [w_i]_{1 \leq i \leq n}$, and the closed-loop $\Phi$ when each $G_i$ is in feedback with its corresponding controller $K_i$. We allow the controller $K_i$ to be connected to the other controllers $K_j$. Thus, the overall controller $K$, given by the relation $u = Ky$, can be a full matrix, where $y$ and $u$ are the vector concatenations of the local measurement and control signals $y_i$ and $u_i$, respectively. For any $K$ that stabilizes the collective of $G_i$’s, the corresponding $\Phi$ can be obtained via a Youla-Kucera parametrization~\cite{youla1976modern} as
\begin{align}\label{eq:phi_map}
    \Phi = w \mapsto z = H - UQV,
\end{align}
where $H = \text{diag}(H_1, \dots, H_n), U = \text{diag}(U_1,\dots, U_n), V = \text{diag}(V_1,\dots, V_n)$ are diagonal stable systems, the elements of which can be obtained from standard factorizations of the individual $G_i$’s. The system $Q$ can be a full matrix of stable systems $[Q_{ij}]_{1 \leq i,j \leq n}$. Given this, we have
\begin{align}\label{eq:phi_matrix}
    \Phi = \begin{bmatrix}
        H_1 - U_1Q_{11}V_1 & -U_1Q_{12}V_2 \dots \\
        -U_2 Q_{21} V_1 & H_2 - U_2Q_{22}V_2 \dots \\
        \vdots & \vdots & \ddots
    \end{bmatrix}.
\end{align}
We are interested in optimizing the performance of the system with respect to a variable that measures deviations from the population average. In this sense, define
\begin{align}\label{eq:error_from_avg}
    e_i := z_i - \bar{z}, \ \ \bar{z} :=  \frac{1}{n} \sum_{j=1}^n z_j, \ \ e := [e_i]_{1\leq i \leq n}. 
\end{align}

\noindent The following map is of interest
\begin{align}\label{eq:psi_map}
    \Psi := w \mapsto e, \ \ \Psi =  \left( I - \frac{1}{n}\mathbf{1}\mathbf{1}^\top \right)\Phi = \mathbf{T}\Phi.
\end{align}

Our objective is to find a controller to minimize some norm of the operator $\Psi$ or a more general problem that involves norms other than deviation from average signals. In particular, if we let $\xi$ be an additional signal of interest, with $\xi = \Xi w$, and the corresponding closed-loop map $\Xi$, we consider
\begin{align}\label{eq:problem_with_Xi}
   \psi^o := \inf_{Q: \|\Xi\| \leq \gamma_\Xi} \|\Psi\|, 
\end{align}
where $\gamma_\Xi$ is a positive constant, the norms in the cost and
constraints are the same for simplicity\footnote{Also, for simplicity, we
do not explicitly denote the dependence of the maps and costs
on the number of agents $n$.}. The map $\Xi$ can capture regular, not necessarily deviation from average, signals such as absolute control actions (not relative to
the collective average). For instance, in the case of stable systems $G_i$ with $\xi = u$ as the signal of interest, $\Xi = Q$, thus, by constraining the absolute control action, we are regularizing problem~\eqref{eq:problem_with_Xi} to avoid singular solutions. In general, we
will have that
\begin{align*}
    \Xi = H_\xi - U_\xi Q V_\xi
\end{align*}
where $H_\xi = \text{diag}(H_{\xi_1}, \dots, H_{\xi_n}), U_\xi = \text{diag}(U_{\xi_1},\dots, U_{\xi_n}), $ $ V_\xi = \text{diag}(V_{\xi_1},\dots, V_{\xi_n})$ are diagonal stable systems, the elements of which can be obtained from standard factorizations of the individual $G_i$’s. The system $Q$ can be a full matrix of stable systems $[Q_{ij}]_{1 \leq i,j \leq n}$. In this article, we will assume that the signals $w_i, \xi_i$, and $z_i$ are scalar signals\footnote{We avoid unnecessary complexity by assuming scalar signals; the case of MIMO problems follows the same path and can be derived analogously}. We consider strongly non-identical agents, where the factorizations involved do not provide identical model matching problems for all agents~\cite{ACC18,tac21}. Here, the norm of interest is the $\ell_2$-induced, i.e., the $\mathcal{H}_\infty$ norm. We impose a regularizing bound $0 < \gamma_i < \infty$ on the diagonal entries $Q_{ii}$ for all $i$ to guarantee that the closed loop is $\ell_2$ stable as $n \to \infty$ (string stability). 

We focus on the following baseline problem, given $\alpha$
\begin{align}\label{eq:original_opt_problem}
   \psi^o := \inf_{Q \in \mathcal{Q}^\alpha} \|\Psi\| = \inf_{Q \in \mathcal{Q}^\alpha} \|\mathbf{T}\Phi\|, 
\end{align}
where $\mathcal{Q}^\alpha := \{Q | Q \ \text{is $\alpha$-column diagonal dominant}, \|Q_{ii}\| \leq \gamma_i < \infty, 1 \leq i \leq n \}$. For a given $M \leq n$, let $\Pi_M$ represent the $M^{th}$ truncation operator i.e. $\Pi_M \Psi = [\Psi_{ij}]_{1 \leq i,j \leq M}$ for $M \leq n$. Then
\begin{align}
    \|\Psi\| & \geq \| \Pi_M \Psi\|, \\
    \Pi_M \Psi & =  \Pi_M \Phi - \Pi_M \frac{1}{n} \mathbf{1}\mathbf{1}^\top \Phi.
\end{align}
$\|\Pi_M \Phi\|$ can be interpreted as capturing the total individual cost of $M$ agents within the aggregate of $n$, while $\|\Pi_M \Psi\|$ captures the total ``social" (deviation from average) cost. We assume the following about the factorization of each $G_i$

\begin{assumption}\label{assmp:bound_on_system_matrices}
    All factors $H_i$, $U_i$ and $V_i$ in~\eqref{eq:phi_matrix} are uniformly
bounded, i.e., for all i
    \begin{align}
        \| H_i \| \leq \gamma_h, \ \ \|U_i \| \leq \gamma_u, \ \ \|V_i \| \leq \gamma_v.
    \end{align}
\end{assumption}

The following lemma shows that, roughly speaking, $\Pi_M \Psi$ remains close to $\Pi_M \Phi$ for a fixed truncation $M$ of the (large) number of agents. In other words, if we consider a block of $M$ agents within an ensemble of $n$ (i.e., $n > M$), then the total individual cost of the agents captured in $\|\Pi_M \Phi \|$ is the one that matters in the total cost $\| \Pi_M \Psi \|$ as $n$ increases. This will be key in proving our main results in the sequel.

\begin{lemma}\label{lem:ensemble_residual_to_zero}
    Let Assumption~\ref{assmp:bound_on_system_matrices} holds and let $\alpha = o(\sqrt{n})$, then under the constraint $Q \in \mathcal{Q}^\alpha$ it holds that 
    \begin{align*}
     \left\|  \Pi_M \frac{1}{n} \mathbf{1}\mathbf{1}^\top \Phi \right\| \to 0, \ \text{as} \ n \to \infty.
    \end{align*}
\end{lemma}
\begin{proof}
Let $ \frac{1}{n} \mathbf{1}\mathbf{1}^\top \Phi := \frac{1}{n} [\mathbf{1} \bar{\Phi}_1 \dots \mathbf{1}\bar{\Phi}_M]$ and $ \Pi_M \frac{1}{n} \mathbf{1}\mathbf{1}^\top \Phi := \frac{1}{n} [\mathbf{1}_M \bar{\Phi}_1 \dots \mathbf{1}_M \bar{\Phi}_M]$, where, $\bar{\Phi}_j = \sum_{i=1}^n \Phi_{ij}$ with $\bar{\Phi}_j = H_j - \sum_{i=1}^n U_i Q_{ij} V_j$. Consider, 
\begin{align}
    \left\|  \sum_{i=1}^n U_i Q_{ij} V_j \right\| &\leq \gamma_j \gamma_u \gamma_v  + \left\| \sum_{i=1, i \neq j}^n U_i Q_{ij} V_j \right \| \nonumber \\
    & \hspace{-0.4in} \leq \gamma_j \gamma_u \gamma_v + \left\| \sum_{i=1, i \neq j}^n U_i Q_{ij} \right \| \|V_j\| \nonumber \\
    & \hspace{-0.4in} \leq \gamma_j \gamma_u \gamma_v +  \sum_{i=1, i \neq j}^n \left\| U_i Q_{ij} \right \| \|V_j\| \nonumber \\
    & \hspace{-0.4in}  \leq \gamma_j \gamma_u \gamma_v +  \sup_{1 \leq i \leq n}\| U_i\|  \sum_{i=1, i \neq j}^n \left\| Q_{ij} \right \| \|V_j\| \nonumber \\
    & \hspace{-0.4in} \leq \gamma_j \gamma_u \gamma_v +  \alpha \sup_{1 \leq i \leq n}\| U_i\|  \left\| Q_{ii} \right \|\|V_j\|\nonumber \\
    & \hspace{-0.4in}  \leq \gamma_j \gamma_u \gamma_v +  \alpha \gamma_i \gamma_{u} \gamma_v \leq \gamma_{Q} \gamma_u \gamma_v + \alpha \gamma_{Q} \gamma_u \gamma_v, \nonumber
\end{align}
where, $\gamma_Q := \max_{1 \leq j \leq n} \gamma_j < \infty$. Therefore, 
\begin{align}
    \left\| \bar{\Phi}_j \right\| & \leq  \|H_j\| + \left\|\sum_{i=1}^n U_i Q_{ij} V_j \right\| \leq \gamma_h + (1 + \alpha) \gamma_{Q} \gamma_u  \gamma_v. \nonumber
\end{align}
\begin{align*}
   \mbox{Hence,} \ \frac{1}{n} \left\| \bar{\Phi}_j \right\| \leq \frac{1}{n}(\gamma_h + (1 + \alpha) \gamma_{Q} \gamma_u  \gamma_v), \ \mbox{which leads to } 
\end{align*}
\begin{align*}
    \left\|  \Pi_M \frac{1}{n} \mathbf{1}\mathbf{1}^\top \Phi \right\| & =  \left\| \frac{1}{n} [\mathbf{1} \bar{\Phi}_1 \dots \mathbf{1}\bar{\Phi}_M] \right\| \\
    &  \leq \frac{\sqrt{M}}{\sqrt{n}}[\gamma_h + (1 + \alpha) \gamma_{Q} \gamma_u  \gamma_v] \\
    &  \leq \frac{\sqrt{M}[\gamma_h + \gamma_{Q} \gamma_u \gamma_v]}{\sqrt{n}} + \frac{\sqrt{M}\gamma_u \gamma_{Q} \gamma_v \alpha}{\sqrt{n}}.
\end{align*}
Since $M$ is finite and $\alpha = o(\sqrt{n})$ we have $\left\| \Pi_M \frac{1}{n} \mathbf{1}\mathbf{1}^\top \Phi \right\| \to 0$ as $n \to \infty$.
\end{proof}

Note that $\Pi_M \frac{1}{n} \mathbf{1}\mathbf{1}^\top$ represents the effect of the ensemble average on a (fixed) block of $M$ agents. The above lemma shows that this effect diminishes as the ensemble grows. Let
\begin{align}\label{eq:diagonal_solution_problem}
    \mu_M := \inf_{ Z \in \mathcal{Q}^\alpha} \| \Pi_M H - \Pi_M UZ\Pi_M V\|,
\end{align}
where $\Pi_M H$, $\Pi_M U$, $\Pi_M V$ in the above defined optimization are the diagonal maps $\Pi_M H := \text{diag}(H_1, \dots, H_M), $ $\Pi_M U := \text{diag}(U_1,\dots, U_M), \Pi_M V := \text{diag}(V_1, \dots, V_M)$. Let $Z^{o,M}$ be the solution to problem~\eqref{eq:diagonal_solution_problem}\footnote{We assume existence to avoid standard technicalities that do not change
the results and replace optimal with arbitrarily close to optimal in the case when existence is not guaranteed.}, and let
\begin{align}\label{eq:optimal_val_diagonal_solution_problem}
    \Phi^{o,M} := \Pi_M H - \Pi_M U Z \Pi_M V.
\end{align}
Since all the systems $\Pi_M H, \Pi_M U, \Pi_M V$are diagonal
maps, $Z^{o,M}$ and $\Phi^{o,M}$ are decentralized (diagonal). Further, $\mu_M$ represents the minimum total ``individual" cost of the block of $M$ agents, and $\mu_M$ is a non-decreasing sequence which is bounded\footnote{
This is obtained by picking a suboptimal $Z = 0$, which leads to the cost $\|\Pi_M H\| \leq \gamma_h$.} by $\gamma_h$. Let $\mu^o := \limsup_{M \to \infty} \mu_M$ and note that $ \mu^o \leq \gamma_h < \infty$ and thus, $ \lim_{M \to \infty} \mu_M = \limsup_{M \to \infty} \mu_M$, i.e. $\mu^o = \lim_{M\to \infty} \mu_M$. 
Let the corresponding social cost map for the above $M$ agents be $\Psi^M := \mathbf{T}_M \Phi^{o,M}$ where, $\mathbf{T}_M := \left( I_M - \frac{1}{M}\mathbf{1}_M\mathbf{1}_M^\top \right)$. The following Theorem shows that $\Psi^M$ in fact describes a solution for problem~\eqref{eq:original_opt_problem} as $n \to \infty$.

\begin{theorem}\label{theorem:H_inf_equivalent}
    Let Assumption~\ref{assmp:bound_on_system_matrices} holds and let $\alpha = o(\sqrt{n})$, then under the constraint $Q \in \mathcal{Q}^\alpha$ it holds that 
    \begin{align*}
       \lim_{n \to \infty} \psi^o = \mu^o,
    \end{align*}
    and arbitrarily close to the optimal decentralized controller can be obtained by $Z^{o,M}$ for sufficiently large $M$.
\end{theorem}
\begin{proof}
    From Lemma~\ref{lem:ensemble_residual_to_zero}, we have that given any fixed $M \leq n$ and any $Q$ satisfying the constraint, $Q\in \mathcal{Q}^\alpha$, selected to form $\Psi$ as in~\eqref{eq:psi_map} and the corresponding $\Phi$ as in~\eqref{eq:phi_map} we have 
    \begin{align*}
        \| \Psi\| \geq \|\Pi_M \Psi \| \geq \| \Pi_M \Phi \| - \varepsilon,
    \end{align*}
    where $\varepsilon \to 0$ as $n \to \infty$. Therefore, $\liminf_{n \to \infty} \| \Psi\| \geq \limsup_{n \to \infty} \| \Pi_M \Phi \|$. However, from the definition of $\mu_M$, for any $Q$ we have $\| \Pi_M \Phi \| \geq \mu_M$, so $\liminf_{n \to \infty} \|\Psi\| \geq \mu_M$ and therefore, 
    \begin{align*}
        \liminf_{n \to \infty} \|\Psi\| \geq \limsup_{n \to \infty} \mu_M = \mu^o.
    \end{align*}
As the above is valid for any sequence of $Q$s we choose, it would also hold for the optimal $Q$ for every $n$, and thus 
\begin{align}\label{eq:dummy}
    \liminf_{n \to \infty} \ \psi^o \geq \mu^o,
\end{align}
which suggests that $\mu^o$ is a lower bound on the optimal asymptotic performance as $n \to \infty$. We now demonstrate that this can be achieved by decentralized control. Indeed, if we let $\Psi^M = \mathbf{T}_M \Phi^{o,M}$, i.e., the mapping we get for $M$ agents by solving for the optimal $Z$ in~\eqref{eq:diagonal_solution_problem}, which leads to a decentralized controller. We get for the corresponding performance of the this $M$ agent system $\|\Psi^M\| \leq \|\mathbf{T}_M\| \| \Phi^{o,M}\|$ or since, $\|\mathbf{T}_M\| = 1$ and $\| \Phi^{o,M}\| = \mu_M$, we have $\|\Psi^M\| \leq \mu_M$. At the same time, as $\Psi^M$ is suboptimal, thus, $ \psi^o \leq \|\Psi^M\|$, where $\psi^o$ is the optimal cost in~\eqref{eq:original_opt_problem} with $M = n$. Hence, 
\begin{align*}
    \liminf_{M \to \infty} \psi^o \leq \limsup_{M \to \infty} \psi^o \leq \limsup_{M \to \infty} \|\Psi^M\| \leq \limsup_{M \to \infty} \mu_M = \mu^o
\end{align*}
or, using~\eqref{eq:dummy} that $\liminf_{M \to \infty} \psi^o \geq \mu^o$ we have that $ \lim_{M\to \infty} \| \Psi^M\| = \mu^o$, which completes the proof. 
\end{proof}

\subsection{(Scaled) $\mathcal{H}_2$ Case}
We now consider $\mathcal{H}_2$ norm minimization. As in the previous section, we focus on the following baseline problem of interest, given $\alpha$
\begin{align}\label{eq:original_h2_opt_problem}
   \psi^o_2 := \inf_{Q \in \mathcal{Q}^\alpha}  \frac{1}{\sqrt{n}} \|\Psi\|_2, 
\end{align}
where we impose the same constraint as in problem~\eqref{eq:original_opt_problem}. We scale the $\mathcal{H}_2$ cost by $\frac{1}{\sqrt{n}}$ where $n$ is the number of agents, as the unscaled cost will become unbounded as $n \to \infty$ preventing any comparison with the decentralized optimal cost. In this way, \eqref{eq:original_h2_opt_problem} is uniformly bounded for all $n$ and $(\psi^o_2)^2$ is interpretable as the average cost (energy) per agent. 

\begin{lemma}\label{lem:ensemble_h2_residual_to_zero}
    Let Assumption~\ref{assmp:bound_on_system_matrices} holds and let $\alpha = o(\sqrt{n})$, then under constraint $Q \in \mathcal{Q}^\alpha$ it holds that 
    \begin{align*}
        \lim_{n \to \infty} \frac{1}{\sqrt{n}} \left\| \frac{1}{n} \mathbf{1}\mathbf{1}^\top \Phi \right\|_2 = 0.
    \end{align*}
\end{lemma}
\begin{proof}
    Consider, 
    \begin{align*}
        \left\| \frac{1}{n} \mathbf{1}\mathbf{1}^\top U Q \right\|_2^2 & =  \frac{1}{n^2} \left \| [\mathbf{1} U_1 \dots \mathbf{1}U_n] Q \right \|_2^2 \\
        & = \frac{1}{n^2} \left \| \begin{bmatrix}
            A \\
            A \\
            \vdots \\
            A
        \end{bmatrix} \right \|_2^2
    \end{align*}
    where,
    \begin{align*}
    A := \left[ \sum_{j=1}^n U_j Q_{j1} \ \sum_{j=1}^n U_j Q_{j2} \dots \ \sum_{j=1}^n U_j Q_{jn} \right]. 
    \end{align*}
    Therefore, 
    \begin{align*}
         \frac{1}{n^2} \left\| \mathbf{1}\mathbf{1}^\top U Q \right\|_2^2 =  \frac{1}{n} \|A\|_2^2 = \frac{1}{n} \sum_{i=1}^n \left \|\sum_{j=1}^n U_j Q_{ji} \right\|_2^2.
    \end{align*}
    We focus on 
    \begin{align*}
        \left \|\sum_{j=1}^n U_j Q_{ji} \right\|_2^2 &=  \left \| U_i Q_{ii} + \sum_{j=1, j \neq i}^n U_j Q_{ji} \right\|_2^2 \\
        &  \hspace{-0.5in} = \| U_i Q_{ii} \|_2^2 + \left \| \sum_{j=1, j \neq i}^n U_j Q_{ji} \right\|_2^2 \\
        &  \hspace{-0.4in} + 2 \|U_i Q_{ii}\|_2\left \| \sum_{j=1, j \neq i}^n U_j Q_{ji} \right\| \\
        & \hspace{-0.5in} \leq \|U_i\|_2^2\|Q_{ii}\|^2 + \left( \sum_{j=1, j \neq i}^n \|U_j Q_{ji}\|_2 \right)^2 \\
        &  \hspace{-0.4in} + 2 \|U_i Q_{ii}\|_2 \left \| \sum_{j=1, j \neq i}^n U_j Q_{ji} \right\|_2 \\
        & \hspace{-0.5in} \leq \sup_{1 \leq i \leq n} \|U_i\|^2\|Q_{ii}\|^2 \\
        &  \hspace{-0.4in} + \sup_{1\leq j \leq n} \|U_j\|^2 \left( \sum_{j=1, j \neq i}^n \| Q_{ji}\| \right)^2 \\
        &  \hspace{-0.4in} + 2 \sup_{1\leq i \leq n} \|U_i\| \|Q_{ii}\| \sup_{1\leq j \leq n} \|U_j\|  \sum_{j=1, j \neq i}^n \| Q_{ji} \| \\
        & \hspace{-0.5in} \leq \gamma_u^2 \gamma_Q^2 ( 1 + \alpha^2 + 2 \alpha ) = \gamma_u^2 \gamma_Q^2 (1+\alpha)^2.  
    \end{align*}
    Therefore, 
    \begin{align*}
         \left\| \frac{1}{n} \mathbf{1}\mathbf{1}^\top U Q \right\|_2^2 \leq \frac{1}{n} \left( n \gamma_u^2 \gamma_Q^2 (1+\alpha)^2 \right) = \gamma_u^2 \gamma_Q^2 (1+\alpha)^2.
    \end{align*}
    Hence,
    \begin{align*}
        \frac{1}{n} \left\| \frac{1}{n} \mathbf{1}\mathbf{1}^\top U Q \right\|_2^2 \leq \frac{\gamma_u^2 \gamma_Q^2 (1+\alpha)^2}{n} = \gamma_u^2 \gamma_Q^2 \left(\frac{1}{\sqrt{n}}+\frac{\alpha}{\sqrt{n}} \right)^2. 
    \end{align*}
    Therefore, $\frac{1}{n} \left\| \frac{1}{n} \mathbf{1}\mathbf{1}^\top U Q \right\|_2^2 \to 0$ as $n \to \infty$. Similarly, due to $\|H\| \leq \gamma_h, \|V\| \leq \gamma_v$
    \begin{align*}
        \frac{1}{n} \left\| \frac{1}{n} \mathbf{1}\mathbf{1}^\top H \right\|_2^2 \to 0 \ \text{and} \ \frac{1}{n} \left\| \frac{1}{n} \mathbf{1}\mathbf{1}^\top U Q V \right\|_2^2 \to 0,
    \end{align*}
    as $n \to \infty$ so that $\frac{1}{\sqrt{n}} \left\| \frac{1}{n} \mathbf{1}\mathbf{1}^\top \Phi \right \|_2 \to 0$. 
\end{proof}
Let
\begin{align}\label{eq:H2_decentalized_problem}
   \mu^o_2 := \inf_{Q \in \mathcal{Q}^\alpha}  \frac{1}{\sqrt{n}} \|\Phi\|_2. 
\end{align}
\begin{theorem}\label{theorem:H_2_equivalent}
    Let Assumption~\ref{assmp:bound_on_system_matrices} holds and let $\alpha = o(\sqrt{n})$, then under the constraint $Q \in \mathcal{Q}^\alpha$ it holds that
    \begin{align*}
        \lim_{n \to \infty} (\psi^o_2 - \mu^o_2) = 0.
    \end{align*} 
    Moreover, the solution to~\eqref{eq:H2_decentalized_problem} is decentralized (diagonal) and it also minimizes~\eqref{eq:original_h2_opt_problem} for large enough $n$.
\end{theorem}
\begin{proof}
Note that
\begin{align*}
    \|\Phi\|_2 - \left \| \frac{1}{n} \mathbf{1}\mathbf{1}^\top \Phi \right \|_2 \leq \|\Psi\|_2 \leq \|\Phi\|_2 + \left \| \frac{1}{n} \mathbf{1}\mathbf{1}^\top \Phi \right \|_2
\end{align*}
or
\begin{align*}
       \frac{-1}{\sqrt{n}} \left \| \frac{1}{n} \mathbf{1}\mathbf{1}^\top \Phi \right \|_2 \leq \frac{1}{\sqrt{n}} \|\Phi\|_2 - \frac{1}{\sqrt{n}} \|\Psi\|_2 \leq \frac{1}{\sqrt{n}} \left \| \frac{1}{n} \mathbf{1}\mathbf{1}^\top \Phi \right \|_2
\end{align*}
which proves that $\lim_{n \to \infty}(\psi^o_2 - \mu^o_2) = 0$ from Lemma~\ref{lem:ensemble_h2_residual_to_zero}. Also note that $\mu^o_2$ (and consequently $\psi^o_2)$ is bounded uniformly in $n$ as $\mu^o_2 \leq \frac{1}{\sqrt{n}}\|H\|_2$ and 
\begin{align*}
   \| H \|_2^2 = \|H_1\|^2_2 + \dots + \| H_n\|^2_2 \leq n \gamma_h^2 
\end{align*}
so that $\mu^o_2 \leq \gamma_h$. The solution $Q^o$ is decentralized (diagonal) follows immediately from the fact that $H$, $U$, and $V$ are decentralized (diagonal).
\end{proof}

\begin{remark}
So far we used $\Xi \equiv Q$ for simplicity
in the constrained problems. Suppose we impose uniform bounds on $\|\Xi_{ii}\|$ for all $i$ and each $n$ instead. In that case, this is equivalent to imposing a uniform bound on $\|Q_{ii}\|$ if $U_{\xi_i}(\lambda)$ and $V_{\xi_i}(\lambda)$ have rank uniformly bounded away from zero on the circle $|\lambda| = 1$. This would be when the standard assumptions for non-singular problems are satisfied~\cite{dahleh1994control}. Similar remarks regarding $\Xi$ apply to the 2-block formulations presented in the following sections.
\end{remark}

\begin{remark}
We emphasize that Assumption~\ref{assmp:bound_on_system_matrices}, which requires a uniform bound on the $\mathcal{H}_\infty$ norm of the collective, serves only as a sufficient condition for establishing the optimality of decentralized and selfish control. Importantly, this requirement is not necessary. In fact, it merely imposes a condition on the coprime factors of each agent $i$. When these factors are identical, the assumption becomes redundant, reducing the problem to the homogeneous-agent case analyzed in \cite{voueli17journal, ACC18}. Moreover, all of the $\mathcal{H}_\infty$ results presented here remain valid for systems with time-varying dynamics, where the relevant norm is the $\ell_2$-induced norm.
\end{remark}

\section{2-block problems}\label{sec:2bloack_problems}
In this section, we derive equivalent results for 2-block problem formulations, starting with the $\mathcal{H}_\infty$ case.
\subsection{ ${\cal H}_\infty$ case}\label{sec:2block_H_inf}
Here, we are interested in solving the problem,  
\begin{equation}\label{eq:2block_H_inf}
\psi^o := \inf_{Q \in \mathcal{Q}^\alpha} \left\| \begin{bmatrix}
    \Psi \\ Q
\end{bmatrix} \right\|.
\end{equation}
Using the same approach we took for problem~(\ref{eq:original_opt_problem}) it holds that for the ${\cal H}_\infty$ case, obtaining 
$Z^{o,M}$ that solves 
\begin{equation}\label{eq:2block-M}
\mu_M := \inf_{ Z \in \mathcal{Z}^\alpha} \left\|  \begin{bmatrix}
\Pi_M\Phi \\ 
Z
\end{bmatrix}
\right\| 
\end{equation}
for sufficiently large $M$, provides a decentralized controller that delivers performance arbitrarily close to $\lim_{n\rightarrow\infty}\psi^o$. In particular, if 
$\displaystyle\mu^o:=\limsup_{M\rightarrow\infty}\mu_M$,
we note that $\mu_M$ is a non-decreasing sequence of $M$ which is bounded by $\gamma_h$ and thus $\displaystyle\lim_{M\rightarrow\infty}\mu_M=\displaystyle\limsup_{M\rightarrow\infty}\mu_M$, i.e.,
$\displaystyle\mu^o=\lim_{M\rightarrow\infty}\mu_M$, and we have the following result.

\begin{theorem}\label{theorem:H_inf_2block}
Let Assumption~\ref{assmp:bound_on_system_matrices} holds and $\alpha = o(\sqrt{n})$, then under the constraint $Q \in \mathcal{Q}^\alpha$ for the $\cal{H}_\infty$ problem~(\ref{eq:2block_H_inf}) we have 
\begin{align*}
    \lim_{n\rightarrow\infty}\psi^o=\mu^o,
\end{align*} and arbitrarily close to the optimal decentralized controller can be obtained by $Z^{o, M}$ for sufficiently large $M$. 
\end{theorem}

\begin{proof} Since for any solution $Q$ it holds that
\begin{align*}
\left\|\begin{bmatrix}
    \Psi \\ Q
\end{bmatrix} \right\|\leq
\left\| \begin{bmatrix}
    \mathbf{T}H \\ 0
\end{bmatrix} \right\| \leq \left\| H\right\|,
\end{align*}
where we used, the fact $\|\mathbf{T}\| = 1$.
Thus, it is enough to search for $Q \in \mathcal{Q}^\alpha$ with $\left\|Q_{ii}\right\| = \gamma_i \le
\left\| H_{i}\right \|\le \gamma_h$ for all $i$.  Hence, everything holds as in the constrained case with $Q \in \mathcal{Q}^\alpha$ i.e.,
\begin{align}
    \left\| \left[\begin{array}{c}
    \Psi\\
    Q\\
    \end{array}
    \right] \right\| &\ge \left\|\left[
    \begin{array}{c}
    \Pi_M\Psi\\
    \Pi_MQ\\
    \end{array}
    \right] \right\|, \\
    \Pi_M \Psi &=  \Pi_M\Phi-\Pi_M\frac{1}{n}\mathbf{1}\mathbf{1}^\top\Phi.
\end{align}
Therefore, as in Lemma~\ref{lem:ensemble_residual_to_zero}, it holds that $\left\|\Pi_M \frac{1}{n}\mathbf{1}\mathbf{1}^\top\Phi\right\|\rightarrow 0$ as $n\rightarrow\infty$.  This, in turn, means that 
$$\left\|\left[
\begin{array}{c}
\Psi\\
Q\\
\end{array}
\right] \right\|  \ge \left\|\left[
\begin{array}{c}
\Pi_M\Psi\\
\Pi_MQ\\
\end{array}
\right] \right\| \ge \left\|\left[
\begin{array}{c}
\Pi_M\Phi\\
\Pi_MQ\\
\end{array}
\right] \right\|-\varepsilon$$
with $\varepsilon>0$ and $\varepsilon\rightarrow 0$ as $n\rightarrow\infty$ and the same approach outlined in Theorem~\ref{theorem:H_inf_equivalent} follows through. 
\end{proof}

\begin{figure*}[t!]
    \centering
    \begin{subfigure}[b]{\linewidth}
        \centering
        \includegraphics[width=0.9\linewidth]{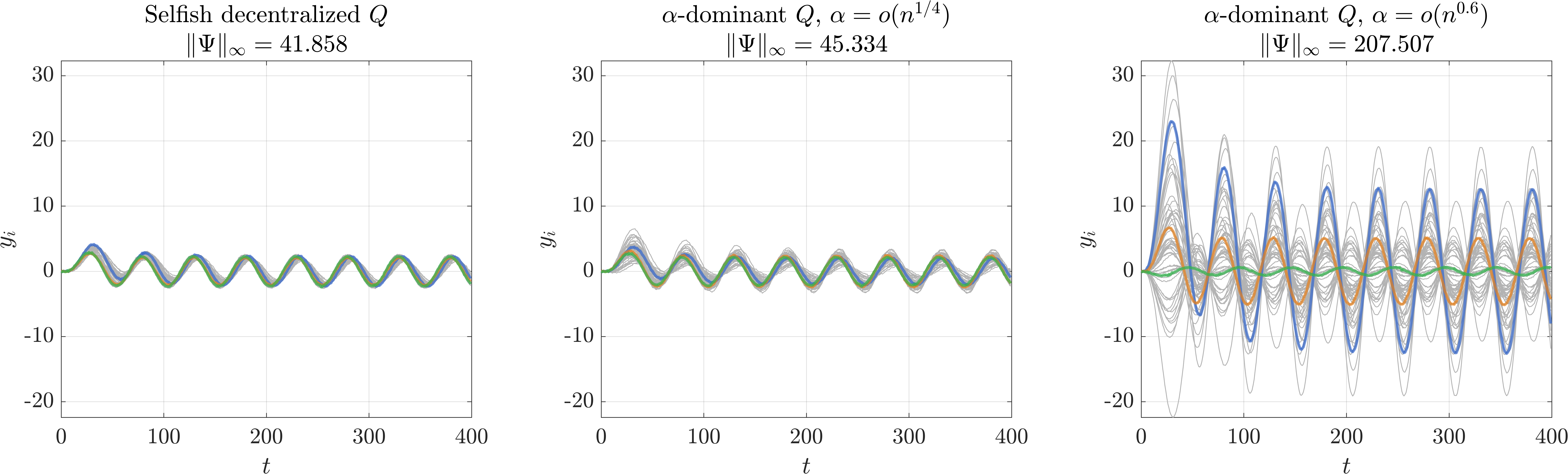}
    \caption{$n =60$}
    \label{subfig:simn60}
    \end{subfigure}
    \par\vspace{0.2cm}
    \begin{subfigure}[b]{\linewidth}
        \centering
        \includegraphics[width=0.9\linewidth]{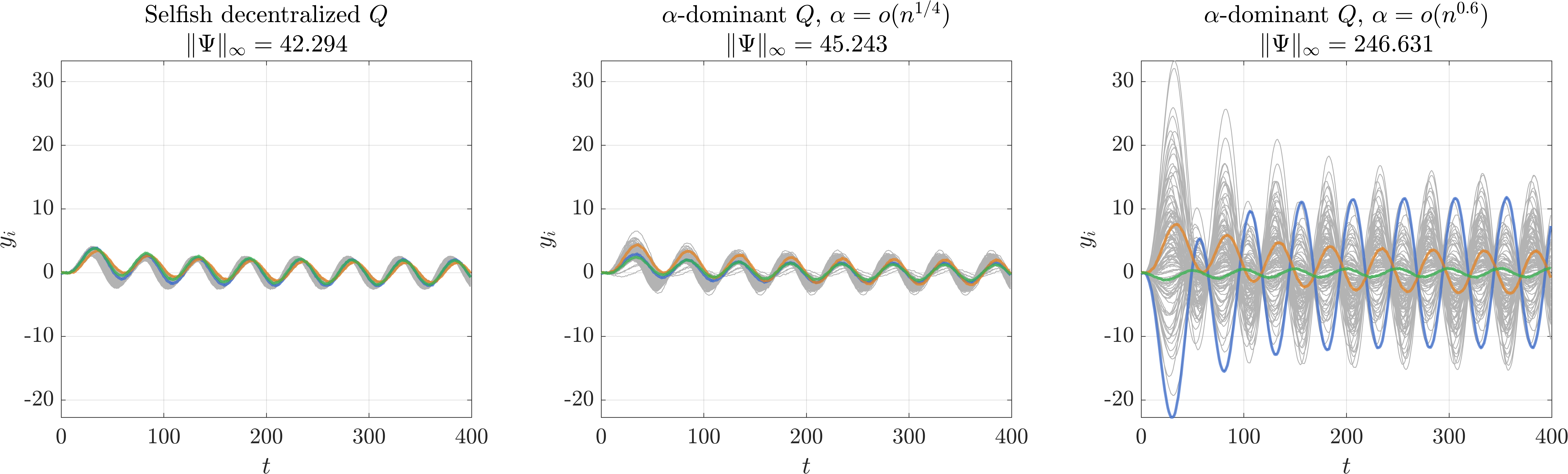}
    \caption{$n =120$}
    \label{subfig:simn120}
    \end{subfigure}
    \caption{Agents' response to noisy inputs with a common sinusoidal signal. Experiment with $n=60,120$ agents. For visualization, the majority of agent responses are depicted in gray, and a randomly selected subset is emphasized in color.}
    \label{fig:sim}
\end{figure*}

\subsection{ ${\cal H}_2$ case}\label{sec:2block_H_2}
For the $\mathcal{H}_2$ case, the results of Theorem~\ref{theorem:H_2_equivalent} remain valid if we maintain a uniform constraint on the $\cal{H}_\infty$ norm $\|Q_{ii}\|$ for all $i$ as in the previous section.  In particular, if 

\begin{equation}\label{eq:2block_H_2}
\psi^o_2:= \inf_{Q \in \mathcal{Q}^\alpha}  \frac{1}{\sqrt{n}}\left\|\left[
\begin{array}{c}
\Psi\\
Q\\
\end{array}
\right] \right\|_2
\end{equation}
it is enough to obtain 
$Q^{o}$ that solves 
\begin{equation}\label{eq:2block-M-2}
\mu^o_2:= \inf_{Q \in \mathcal{Q}^\alpha}  \frac{1}{\sqrt{n}}\left\|\left[
\begin{array}{c}
\Phi\\
Q\\
\end{array}
\right] \right\|_2  
\end{equation}
for sufficiently large $n$, which defines a decentralized controller that delivers performance arbitrarily close to $\psi^o_2,$ i.e.,  

\begin{theorem}\label{theorem:H2_2block}
Let Assumption~\ref{assmp:bound_on_system_matrices} holds and $\alpha = o(\sqrt{n})$, then under the constraint $Q \in \mathcal{Q}^\alpha$ for the $\mathcal{H}_2$ problem~\eqref{eq:2block_H_2} we have 
\begin{align*}
    \lim_{n\rightarrow\infty}(\psi^o_2-\mu^o_2)=0,
\end{align*}
an arbitrarily close to the optimal decentralized controller can be obtained by the solution of the problem~\eqref{eq:2block-M-2} for sufficiently large $n$.
\end{theorem}

\begin{proof} Note that 
\begin{align*}
\displaystyle\left[
\begin{array}{c}
\Psi\\
Q\\
\end{array}
\right] =\left[
\begin{array}{c}
\Phi\\
Q\\
\end{array}
\right] +\left[
\begin{array}{c}
-\frac{1}{n}\mathbf{1}\mathbf{1}^\top \Phi\\
0\\
\end{array}
\right]. 
\end{align*}
As $Q \in \mathcal{Q}^\alpha$, from Lemma~\ref{lem:ensemble_h2_residual_to_zero},
\begin{align*}
    \frac{1}{\sqrt{n}} \left\|\left[
\begin{array}{c}
\frac{1}{n}\mathbf{1}\mathbf{1}^\top \Phi\\
0\\
\end{array}
\right] 
\right\|_2
\rightarrow 0 \text{ as }n\rightarrow\infty,
\end{align*} 
which proves the assertion.
\end{proof}

In Theorem \ref{theorem:H2_2block}, the imposed constraint $\left\|Q_{ii}\right\|\le\gamma_i$ for all $i$ (as $Q \in \mathcal{Q}^\alpha$) for the  ${\cal H}_2$ problem, where $\gamma_i$ can be arbitrarily large. This is to ensure closed-loop stability for the ``infinite" ensemble, i.e., as $n\rightarrow\infty.$. This comes for free in the ${\cal H}_\infty$ problem in Theorem \ref{theorem:H_inf_2block} as for the solution $Q^o$ of~\eqref{eq:2block-M} we have that $\left\|Q^{o}_{ii}\right\|\le\gamma_h$ for all $i$ and any $n$.

\section{Case Study}
Here, we design a simulation case study to verify our theoretical results. We consider heterogeneous agents with each agent $i$ described by the following discrete-time dynamics:
\begin{align*}
x_{i_1}(k+1) &= x_{i_1}(k) + x_{i_2}(k), \\
x_{i_2}(k+1) &= a_i x_{i_2}(k) + w_i(k) + b_i u_i(k), \\
y_i(k) &= -x_{i_1}(k) + v_i(k),
\end{align*}
where the parameters $a_i \in [0.5,\,1.5]$ and $b_i \in [0.8,\,1.2]$ are drawn independently from uniform distributions, ensuring heterogeneity between agents. The disturbance $v_i(k)$ is modeled as a sinusoidal signal, while the measurement noise $w_i(k)$ is a Gaussian random variable with mean zero and a standard deviation $0.05$. The global reference trajectory is given by a sinusoidal signal. 
We simulate systems with $n = 60$ and $n = 120$ agents. For each agent, we compute the selfish decentralized controller that minimizes the $\mathcal{H}_\infty$ from $[w_i \ v_i]^\top$ to $[z_i \ \xi_i ]^\top$. Here, the cost focuses on minimizing the deviations from the average of all agents (measured via $z_i$) while keeping the control input ``small''. We compare the decentralized solution with the controller designed while considering the $\alpha$-diagonal dominance constraints for $\alpha = o(n^{1/4})$ (satisfying the conditions in Theorem~\ref{theorem:H_inf_equivalent}) and $\alpha = o(n^{0.6})$ (violating the conditions in Theorem~\ref{theorem:H_inf_equivalent}). Unlike the decentralized controllers, which are obtained by solving each agent’s control problem independently, the $\alpha$-dominant controllers in our simulations are not solutions to any optimization problem. Instead, they are generated by modifying the selfish diagonal solution: for each agent $i$, the diagonal entry $Q_{ii}$ is retained, but part of the norm $\|Q_{ii}\|$ is redistributed across the off-diagonal terms of the same column to enforce (or deliberately violate) the $\alpha$-column diagonal dominance condition. Concretely, we scale $\|Q_{ii}\|$ by the prescribed $\alpha$ and randomly allocate some non-zero norm to the off-diagonal entries such that the resulting $Q$ satisfies the $\alpha$-dominant property. We create a \textit{compliant} case by choosing $\alpha = o(n^{1/4})$, and a \textit{violation} case with $\alpha = o(n^{0.6})$, where the constraint is intentionally broken. These $\alpha$-dominant controllers act as test cases for assessing the sharpness of the theoretical bounds, rather than as new optimal designs.

The $\mathcal{H}_\infty$ norms for the decentralized selfish solution with $n = 60, 120$ agents are $41.858$ and $42.294$ respectively. While the corresponding controller designed under the $\alpha$-diagonal dominance constraints has $\mathcal{H}_\infty$ norms of $45.334$ and $45.243$, respectively for $\alpha = o(n^{1/4})$, and $207.507$ and $246.631$, respectively for $\alpha = o(n^{0.6})$. Fig.~\ref{fig:sim} illustrates the trajectories of the system outputs 
$y_i$ for all agents over a horizon of $400$ discrete-time steps. The decentralized closed-loop responses exhibit a natural tendency to ``reasonably'' follow the sinusoidal input, while achieving a near-optimal norm with respect to deviations from the collective average. In contrast, the controller designed under the $\alpha$-diagonal dominance constraints is, by construction, insensitive to averages. Consequently, its ability to track the common sinusoidal input is weaker. Furthermore, we observe that as the number of agents $n$ increases, the decentralized $\mathcal{H}_\infty$ norm closely tracks the optimal $\mathcal{H}_\infty$ norm achieved via the $\alpha$-diagonally dominant design with $\alpha = o(n^{1/4})$. However, the violation of the dominance condition, when $\alpha =o(n^{0.6})$, leads to a significant difference between the optimal $\mathcal{H}_\infty$ norm and the decentralized solution, suggesting that the conditions in Theorems~\ref{theorem:H_inf_equivalent} and~\ref{theorem:H_2_equivalent} are tight. 

Table~\ref{tab:hinf_results} reports the $\mathcal{H}_\infty$ norm values for varying numbers of agents. The results show that the selfish decentralized controller achieves nearly constant performance across all $n$, the \textit{compliant} $\alpha$-dominant case ($ \alpha = o(\sqrt{n})$ remains close to this baseline decentralized controller, while the \textit{violation} case $ \alpha = o(n^{0.6})$ steadily diverges away from the decentralized cost. These trends reinforce the sharpness of our theoretical guarantees.

\begin{table}[h]
    \centering
    \caption{$\mathcal{H}_\infty$ norm of the closed-loop with different controllers for an increasing number of agents $n$.}
    \label{tab:hinf_results}
    \begin{tabular}{c|c|c|c}
        \hline
        $n$ & Selfish diagonal $Q$ & $\alpha = o(n^{1/4})$ & $\alpha = o(n^{0.6})$ \\
        \hline
        30 & 41.710 & 49.338 & 205.602 \\
        60 & 41.858 & 45.334 & 207.507 \\
        120 & 42.294 & 45.243 & 246.631 \\
        200 & 42.318 & 44.389 & 256.725 \\
        300 & 42.323 & 43.794 & 272.788 \\ 
        400 & 42.326 & 43.486 & 293.791 \\ 
        600 & 42.333 & 43.351 & 319.620 \\
        \hline
    \end{tabular}
\end{table}

\section{Concluding Remarks}
We developed an input–output framework for the design and analysis of decentralized controllers in large-scale multi-agent systems with heterogeneous agents coupled through a mean-field type social cost. We extended the focus of networked control approaches to a setting where each agent’s local input–output map is bounded and the collective dynamics satisfy a diagonal dominance property. In this case, we proved that selfish behavior is asymptotically socially optimal as $n \to \infty$ in the case of $\mathcal{H}_\infty$ and $\mathcal{H}_2$ norms if the column diagonal dominance of the collective dynamics is $o({\sqrt{n}})$.    

\bibliographystyle{IEEEtran}
\bibliography{References}

\end{document}